\documentclass[1p]{elsarticle}
\usepackage{amsmath, amssymb, amsthm}

\newtheorem{theorem}{Theorem}

\newtheorem{lemma}[theorem]{Lemma}
\newtheorem{proposition}[theorem]{Proposition}
\newtheorem{corollary}[theorem]{Corollary}

\theoremstyle{definition}
\newtheorem{definition}[theorem]{Definition}
\newtheorem{remark}[theorem]{Remark}
\newtheorem{example}[theorem]{Example}
\DeclareMathOperator{\GL}{GL}

\title{Invariants of Quadratic Forms and applications in Design Theory} 
\author[1]{Oliver W.\ Gnilke}\ead{owg@math.aau.dk}
\author[2]{Padraig \'O Cath\'ain}
\ead{padraig.ocathain@dcu.ie}
\author[3]{Oktay Olmez}
\ead{olmezoktay@gmail.com}
\author[4]{Guillermo Nu\~nez Ponasso}
\ead{gcnunez@wpi.edu} 
\affiliation[1]{organization={ Department of Mathematical Sciences, Aalborg University},
postcode={9220 Aalborg}, 
country={Denmark}}
\affiliation[2]{organization={Fiontar agus Scoil na Gaeilge, Dublin City University},
postcode={Dublin 9},
country={Ireland}}
\affiliation[3]{organization={Afiniti},
postcode={Hamilton},
country={Bermuda}}
\affiliation[4]{organization={Worcester Polytechnic Institute},
postcode={Worcester, MA},
country={United States of America}}
\begin{document}
\thispagestyle{empty}

\begin{abstract} 

The study of regular incidence structures, such as projective planes and symmetric block designs, 
is a well established topic in discrete mathematics. The work of Bruck, Ryser and Chowla in the mid-twentieth 
century applied the Hasse-Minkowski local-global theory for quadratic forms to derive non-existence results 
for certain design parameters. 

Several combinatorialists have provided alternative proofs of this result, replacing conceptual arguments with 
algorithmic ones. In this paper, we show that the methods required are purely linear-algebraic and 
are no more difficult conceptually than the theory of the Jordan Canonical Form. Computationally, they are rather easier.
We conclude with some classical and recent applications to design theory, including a novel application to the 
decomposition of incidence matrices of symmetric designs.
\end{abstract} 
\begin{keyword}
  quadratic forms \sep incidence matrix \sep symmetric design  \MSC[2020] 05B20 \sep 15A63
\end{keyword}

\maketitle
\section{Introduction}

An \textit{incidence structure} consists of \textit{points} and \textit{lines} together with an \textit{incidence relation}, which specifies which points are incident with which lines. A finite incidence structure is represented by a $\{0,1\}$-matrix with rows labeled by lines and columns by points which contains a $1$ in position $(i,j)$ if point $p_{j}$ is incident with line $L_{i}$ and a $0$ otherwise. Very often, a combinatorial regularity is imposed on such a structure by requiring that each line contain a fixed number $k$ of points and each pair of points lie on $\lambda$ lines. In this case, it is easily seen that incidence structure, now called a $2$-design, obeys the equation 
\[ M^{\top}M = (k-\lambda) I_{v} + \lambda J_{v} \,\]
for non-negative integers $v,k,\lambda$, where $J_{v}$ is the $v\times v$ all-ones matrix. Similarly, a strongly regular graph is described by a (symmetric) adjacency matrix having the property that 
\[ M^{\top}M = k I_{v} + \lambda M + \mu (J_{v} - I_{v} - M) \,, \] 
where $v,k,\lambda,\mu$ are the parameters associated with the strongly regular graph. Many other combinatorial 
structures admit a description in terms of a \textit{Gram} matrix, that is a matrix of the form $M^{\top}M$. 
The study of such incidence structures is the central topic of the theory of combinatorial designs, on which many books have 
been written, \cite{BJL, DRG, HughesPiper}.

The study of quadratic forms, particularly over the integers and rational numbers, is a classical topic in number theory. 
A quadratic form is a generalization of an inner product (defined over an arbitrary field); the most 
convenient description is as follows.
\begin{definition} 
Let $V$ be a vector space with a fixed basis $B$. With respect to this basis, a quadratic form is described by a 
symmetric matrix $Q$: 
\[ q(v) = v^{\top} Q v\,.\] 
Two quadratic forms are \textit{similar} if they are equal up to a change of basis. Thus, similar forms are related by 
\[ q_{2}(Av) = q_{1}(v) \] 
for a fixed invertible matrix $A$. The corresponding matrices satisfy $A^{\top}Q_{2}A = Q_{1}$, and are said to be \textit{congruent}.
\end{definition} 

Congruence is an equivalence relation on the set of square matrices over a fixed field, 
and each congruence class describes a unique quadratic form. The fundamental problem in 
classifying quadratic forms is describing the congruence classes of square matrices over a given 
field. This classification depends essentially on the structure of the group of squares of the underlying field. 
Every non-zero element of $\mathbb{C}$ is a square, and it turns out that the rank is the only invariant of a 
quadratic form in dimension $n$ over $\mathbb{C}$. There are infinitely many classes of quadratic forms 
in each dimension over $\mathbb{Q}$, with necessary and sufficient conditions for the equivalence of forms described 
by the Hasse-Minkowski theorem. Historically, due to close connections to representations of integers as sums of 
squares, this is considered a classical topic in number theory, \cite{OMeara, Scharlau,Serre}.

The defining equation $M^{\top}I_{v}M = (k - \lambda) I_{v} + J_{v}$ for a symmetric $2-(v,k,\lambda)$ design 
may be interpreted as a statement about quadratic forms: it says that the matrices $I_{v}$ and 
$(k - \lambda) I_{v} + J_{v}$ are congruent. The contrapositive statement is more interesting: if these matrices can be shown to belong to different congruence classes, no design with the parameters $2-(v,k,\lambda)$ can exist. The original motivation of Bruck and Ryser was to show that there are infinitely many parameter sets where no projective plane can exist. The result was later extended in collaboration with Chowla to cover all symmetric designs.

\begin{theorem}[Bruck-Ryser-Chowla]
Suppose that $(v,k,\lambda)$ are the parameters of a symmetric block design. 
\begin{enumerate} 
\item If $v$ is even then $k-\lambda$ is the square of an integer. 
\item If $v$ is odd then the Diophantine equation 
\[ x^{2} - (k-\lambda)y^{2} - (-1)^{\frac{v-1}{2}}\lambda z^{2} = 0 \] 
has a non-trivial solution in integers $x,y,z$.
\end{enumerate}
\end{theorem} 

The Bruck-Ryser-Chowla theorem was revolutionary in introducing algebraic techniques to design theory, 
though these were not universally adopted by combinatorialists. This (standard) formulation 
does not reference quadratic forms or their invariants. In fact, a series of papers were devoted to 
finding less conceptual and more explicit algorithmic and computational proofs of the Bruck-Ryser-Chowla theorem\footnote{Marshall Hall's foundational text \textit{Combinatorial Theory} may have something to do with this: he adopts the rather cumbersome language of linear forms in Chapter 10, and refers to the \textit{detailed and troublesome calculation} required to establish rational equivalence of certain forms. We aim to show that the language of linear algebra renders these calculations rather straightforward, once the underlying notations and concepts are established.}. In the opinion of the authors, these had the effect of obscuring the underlying ideas of the proof. Testing congruence of matrices is computationally rather easier than testing conjugacy, which requires the Jordan 
Canonical Form. At a high level, this is because the JCF requires manipulation of the characteristic polynomial, 
which may have irreducible factors of arbitrarily high degree, while working with quadratic forms requires only (multi-variate) 
quadratic functions, for which additional methods are available. 

In the first half of this paper, we give a self-contained and relatively elementary construction 
for a set of invariants for quadratic forms. In Section 2, we give some standard definitions and background 
before classifying the quadratic forms in dimension $2$ over $\mathbb{Q}$ by elementary means. We conclude the section 
with a discussion of Legendre and Hilbert symbols, and show that the classification in dimension $2$ is entirely captured by 
the Hilbert symbol. In Section $3$, we follow arguments due originally to Pall to construct invariants of congruence classes of matrices in an arbitrary dimension as a product of Hilbert symbols, \cite{JonesB, Serre}. The proof requires manipulation of matrix minors and Hilbert symbols, 
but is entirely elementary. In fact, the content of the Hasse-Minkowski Theorem is that, together with some obvious invariants (such as the square-free part of the determinant), the set of invariants constructed by Pall is \textit{complete}: two quadratic forms having the same value at all invariants are necessarily similar. This proof does require the $p$-adic integers in an essential way, and is more profound than the construction of the invariants alone. 
Our paper is entirely self contained: a complete proof is given of the results that we require. 

In Section 4, we prove the theorems of Bruck-Ryser and Bruck-Ryser-Chowla. We also 
discuss applications of this theory to group-divisible designs (the Bose-Connor theorem) 
and to maximal determinant matrices. While our proof of the Bruck-Ryser theorem is not the most elegant possible, it may be the most straightforward. It depends only on the computation of invariants of certain quadratic forms, which is purely algorithmic and can be carried out by hand. The application to the decomposition of symmetric designs is new.

Since our intended audience consists of combinatorialists, we make the following assumptions (which 
are always satisfied in our applications). 

\begin{enumerate} 
	\item Vector spaces are finite dimensional over $\mathbb{Q}$, and have a specified basis. 
	Matrices are square with rational entries, and expressed with respect to the specified basis. 
	The transpose operator imposes the standard inner product (a bilinear form) 
	on the underlying vector space. 
	\item The matrix $S$ of a quadratic form will always be assumed symmetric positive definite.	
	By Sylvester's criterion, all principal minors of $S$ can be assumed to be strictly positive. 
	This assumption circumvents several technical complications in the more general case caused by vanishing minors. 
	\item Rather than deciding the equivalence of two arbitrary forms, we will deal with the simpler case of 
	deciding whether $S$ and the quadratic form represented by the identity matrix are equivalent. 
	Testing congruence to the identity matrix is precisely equivalent to deciding whether 
	$S$ is a \textit{Gram matrix}, which is the main concern 
	in applications to design theory. In particular, we will assume that the determinant of $S$ is the 
	square of an integer. In the language of quadratic forms, the square-free part of the 
	determinant is called the \textit{discriminant}, and we assume this is $1$. 	
\end{enumerate} 

Since extensive literature on the Bruck-Ryser-Chowla theorem exists, the need for another 
discussion should perhaps be justified. Many textbooks on number theory 
include complete proofs of the Hasse-Minkowski theorem. These are often followed 
with a remark that the Bruck-Ryser-Chowla theorem on the non-existence of certain 
projective planes follows immediately. In the combinatorics literature, many authors follow 
the treatment of Hall: after giving a few arbitrary-seeming rules 
for manipulating quadratic forms with heavy emphasis on Witt's Cancellation Lemma 
and the Lagrange Four-Square theorem, a process is sketched to allow reduction of the 
forms to dimension at most three, to which an \textit{ad hoc} argument is applied. The final 
statement is then written in terms of Diophantine equations, obscuring the method of the 
proof. An alternative proof by Lander constructs a self-dual code from the incidence matrix 
of a design. It uses the classification of quadratic forms over finite fields to recover most of 
the non-existence results obtained by Bruck-Ryser-Chowla, \cite{Lander}.

We provide a statement in terms of Hilbert symbols which is no more complex than the Diophantine 
equation formulation. It is more tractable computationally and will allow the interested reader 
to apply the theory to their applications. We illustrate this with several applications in the final 
section of the paper.

\section{Background, Hilbert symbols and Dimension 2} \label{Background}

This material is, of course, well known. However, it is often presented in rather greater generality (as the theory of bilinear 
forms over an arbitrary field) within a more advanced graduate course or in lesser generality (as the theory of inner product 
spaces) in a slightly less advanced course. We aim for a middle ground: where it does not lead to additional complication we 
state the theory in more general terms, but our applications will always require that $k = \mathbb{Q}$.

Let $V$ be a finite dimensional vector space of dimension $n$ over a field $k$, of characteristic different 
from $2$. It will be convenient to fix a basis, $B$, of $V$. This gives an explicit isomorphism between elements 
of $\textrm{End}(V)$ and the matrix algebra $\textrm{M}_{n}(k)$. To be entirely clear: if a linear transformation $T$ is defined by $Tb_{j} = \sum_{i=1}^{n} t_{ij}b_{j}$ then the matrix of $T$ is $[t_{ij}]_{i,j}$ where $i,j \in [1, \ldots, n]$. As defined in the introduction, a quadratic form $Q: V \rightarrow k$ is a function $Q: v \mapsto v^{\top} Av$ for some matrix $A \in M_{n}(k)$. The matrix $A$ \textit{represents} the quadratic form, and the pair $(V, Q)$ is a \textit{quadratic space}. The form is \textit{non-degenerate} if $A$ has full rank. 

Provided that the field is not of characteristic $2$, the matrix $A$ may be taken to be symmetric, 
for $Q(v) = v^{\top}Av = v^{\top}A^{\top}v$ implies that $Q(v) = v^{\top}\left( \frac{1}{2} A + \frac{1}{2}A^{\top}\right)v$. 
We will always assume that $A$ is symmetric. Associated to $Q$ there is a \textit{bilinear form} defined by 
\[ \langle x,y\rangle = \frac{1}{2} \left( Q(x+y) - Q(x) - Q(y) \right) \,.\] 
Working in terms of matrices, $\langle x,y \rangle = x^{\top}Ay$. From a symmetric bilinear form, 
we recover a quadratic form as $Q(v) = \langle v,v\rangle$; these concepts are equivalent 
in characteristic different from $2$. We will deal almost exclusively with the quadratic forms in this text. 

Provided that the field satisfies $\mathbb{Q} \subseteq k \subseteq \mathbb{R}$, a form is defined to be 
\textit{positive definite} if $Q(x) > 0$ for all non-zero $x \in V$. A form is positive definite if and only if 
all eigenvalues of an associated matrix $A$ are positive. Like the standard inner product, represented by 
the matrix $I_{n}$, these can be considered functions 
that assign lengths to vectors, just as the corresponding bilinear form describes the angle between 
vectors by setting the cosine of the angle $\angle(u,v)$ to be $\langle u,v\rangle Q(u)^{-1/2}Q(v)^{-1/2}$. 

There are multiple normal forms to which symmetric matrices can be reduced. For our purposes, the following 
reduction will suffice. 

\begin{proposition} \label{RowOps}
Suppose that $S$ is a symmetric matrix defined over $\mathbb{Q}$. There exists a
matrix $M$ with entries in $\mathbb{Q}$ such that $S' = M^\top SM$ is a diagonal matrix. Furthermore, the entries of $S'$ 
can be taken to be square-free integers arranged in increasing order.
\end{proposition} 

\begin{proof} 
If $X = RY$ for matrices of compatible sizes, then the rows of $X$ are linear combinations of the rows of $Y$. 
Equivalently, $X$ is obtained from $Y$ by a sequence of row operations. Similarly, if $X = YC$ then the columns of 
$X$ are linear combinations of the columns of $Y$. 

By the theory of the row echelon form, there exists a matrix $M^{\top}$ which places the symmetric matrix $S$ in 
upper triangular form. Since $S$ is symmetric $SM = (M^{\top}S)^{\top}$ is lower triangular. But $(M^{\top}S)M = 
M^{\top}(SM)$ by associativity of matrix multiplication, so row operations and column operations commute. As a result 
the matrix $M^{\top}SM$ is both upper and lower triangular, and thus is diagonal. 

Since $S$ is a rational matrix, all row and column operations are carried out over $\mathbb{Q}$ and the matrix $M^{\top}SM$ 
is also rational. Let $d$ be the common denominator of the matrix entries. Then $(dI)M^{\top}SM(dI)$ is a matrix congruent to $M$ 
with integer entries. It is clear that conjugation by a permutation matrix places the diagonal elements in arbitrary order, and multiplying the $i^{\textrm{th}}$ row and column by $t^{-1}$ eliminates any term $t^{2}$ on the diagonal, so that the diagonal entries can be taken to be 
square-free integers. 
\end{proof} 

\begin{example}\normalfont
Though row reduction is familiar to any student of linear algebra, we illustrate for clarity the process of polarisation of a matrix by simultaneous row and column reduction. Consider the matrix 
\[ S = \begin{bmatrix} 1 & 2 & 3 \\ 2 & 4 & 5 \\ 3 & 5 &-1 \end{bmatrix}\,. \] 
We begin by eliminating the off-diagonal entries in the first row, and then the first column. 
\[ \begin{bmatrix} 1 & 0 & 0 \\ -2 &1 & 0 \\ -3 & 0 &1 \end{bmatrix}
\begin{bmatrix} 1 & 2 & 3 \\ 2 & 4 & 5 \\ 3 & 5 &-1 \end{bmatrix}
\begin{bmatrix} 1 & -2 & -3 \\ 0 & 1 & 0 \\ 0 & 0 &1 \end{bmatrix} = 
\begin{bmatrix} 1 & 0 & 0 \\ 0 & 0 & -1 \\ 0 & -1 &-10 \end{bmatrix}\,.\]
Since we have a zero pivot in position $(2,2)$, we swap the second and third rows. We can also multiply the second row and column by $10$, to achieve the matrix 
 \[\begin{bmatrix} 1 & 0 & 0 \\ 0 & 0 & 1 \\ 0 & 10 &0 \end{bmatrix}
 \begin{bmatrix} 1 & 0 & 0 \\ 0 & 0 & -1 \\ 0 & -1 &-10 \end{bmatrix}
 \begin{bmatrix} 1 & 0 & 0 \\ 0 & 0 & 10 \\ 0 & 1 &0 \end{bmatrix} = \begin{bmatrix} 1 & 0 & 0 \\ 0 & -10 & -10 \\ 0 & -10 &0 \end{bmatrix}\]
 Finally, subtracting the second row from the third and likewise for columns leaves the diagonal matrix $\langle1, -10, 10\rangle$. The matrix $X$ such that $X^{\intercal}SX=\langle 1,-10,10\rangle$ can be computed explicitly by multiplying out the row operation matrices. 
\end{example}

 Our convention is that the quadratic form represented by the diagonal matrix with entries $a_{1}, a_{2}, \ldots, a_{n}$ is denoted $\langle a_{1}, a_{2}, \ldots, a_{n}\rangle$. As with row operations, the above process is not canonical: it involves choice. Unlike the reduced row echelon form, the result depends on the choices made. For example, 
 \[ \begin{bmatrix} 1 & 0 & 0 \\ 0 & 1 & 3 \\ 0 & -3 & 1 \end{bmatrix} 
 \begin{bmatrix} 1 & 0 & 0 \\ 0 & -1 & 0 \\ 0 & 0 & 1 \end{bmatrix} 
 \begin{bmatrix} 1 & 0 & 0 \\ 0 & 1 & 3 \\ 0 & -3 & 1 \end{bmatrix} = 
 \begin{bmatrix} 1 & 0 & 0 \\ 0 & -10 & 0 \\ 0 & 0 & 10 \end{bmatrix}\,, \] 
 so the matrix $S$ of the example above also represents the quadratic form $\langle 1, -1, 1\rangle$. Note that 
 neither $\langle 1, -10, 10\rangle$ nor $\langle 1, -1, 1\rangle$ give the eigenvalues of $S$, which has characteristic 
 polynomial $\lambda^{4} - 4\lambda^{2} - 39\lambda + 1$, an irreducible cubic with roots at approximately 
 $8.5, -4.5, 0.02$. All matrices representing this form will have two positive and one negative eigenvalue; the number of 
 positive, negative, and zero eigenvalues is called the \textit{signature} of the form.  
 
\begin{remark} 
Over $\mathbb{R}$, any positive definite form can be reduced to diagonal form by Proposition \ref{RowOps}. 
Since every positive real number has a square root, all coefficients can be taken to be $1$. Hence, an inner product 
space over $\mathbb{R}$ carries the standard quadratic form, and one is justified in speaking of \textbf{the} 
$n$-dimensional inner product space over $\mathbb{R}$. As the set of positive non-squares in $\mathbb{Q}$ is richer, so is the theory of rational quadratic forms.
\end{remark} 
 
The theory of quadratic forms is to a large extent the resolution of quadratic forms, represented as diagonal matrices, into equivalence classes. In the next subsection, we completely classify quadratic forms over $\mathbb{Q}$ in dimension $2$ by elementary means. 

\subsection{Classification of quadratic forms over $\mathbb{Q}$ in dimension $2$} 

Clearly, one-dimensional forms $\langle a\rangle$ and $\langle b \rangle$ are similar if and only if $ab$ is a square. 
In particular, the forms equivalent to $\langle 1\rangle$ are precisely those given by rational squares.
The first interesting case is the classification of quadratic forms in dimension $2$. We will describe all forms equivalent to $\langle 1,1\rangle$. 

\begin{proposition} \label{Fermat}
For prime $p$, the quadratic forms $\langle p,p\rangle$ and $\langle 1,1\rangle$ 
are similar if and only if $p$ is a sum of two squares. 
\end{proposition} 

\begin{proof} 
Suppose that $p = a^{2} + b^{2}$. Then 
\[ \begin{bmatrix} a & b \\ -b & a \end{bmatrix} \begin{bmatrix} 1 & 0 \\ 0 & 1 \end{bmatrix} 
\begin{bmatrix} a & -b \\ b & a \end{bmatrix} = \begin{bmatrix} p & 0 \\ 0 & p \end{bmatrix}\,. \] 

Conversely, suppose that $\langle p,p \rangle$ is a quadratic form equivalent to $\langle 1, 1 \rangle$. Then the equation 
\[ \begin{bmatrix} a & b \\ a' & b' \end{bmatrix}\begin{bmatrix} a & a' \\ b & b' \end{bmatrix} = \begin{bmatrix} p & 0 \\ 0 & p \end{bmatrix}\]
has a solution for a rational valued matrix. Equating the top-left entries, we find that $a^{2} + b^{2} = p$. 
Set $a = A/d$ and $b = B/d$ for integers $A$ and $B$, not both even. This yields an integer equation $A^{2} + B^{2} = pd^{2}$. The left hand side evaluates to $1$ or $2$ modulo $4$. Hence $d$ is odd, and $p$ is not $3 \mod 4$. 
\end{proof} 

The following classical result of Fermat classifies the primes arising in Proposition \ref{Fermat} (see, for example, Chapter 8 of Ireland and Rosen, \cite{IrelandRosen}).

\begin{theorem} \label{FermatThm}
A prime is a sum of two (integer) squares if and only if $p = 2$ or $p \equiv 1 \mod 4$. 
\end{theorem}

The integers which are a sum of two squares may be characterised precisely. 
The Brahmagupta-Fibonacci identity, 
\begin{equation} (a^{2} + b^{2})(c^{2} + d^{2}) = (ac-bd)^{2} + (ad+bc)^{2}\,, \end{equation}
shows that the product of sums of two squares is again a sum of two squares. 
Write $n = mr^{2}$, where $m$ is square-free. Clearly, $r^{2} +0$ is a sum of two squares, 
while by Theorem \ref{FermatThm} and Brahmagupta-Fibonacci, $m$ (and hence $n$) 
may be written as a sum of two squares if its prime factors are $2$ or congruent to $1 \mod 4$. In fact, this condition 
is necessary and sufficient. That is, $n$ is a sum of two squares if and only if no prime congruent to $3 \mod 4$ divides its square-free part, \cite[Theorem 366]{HardyWright}. From this discussion, it is not difficult to describe all quadratic forms in dimension two which are equivalent to $I_{2}$. 

\begin{corollary}
In dimension $2$ a quadratic form is congruent to $\langle 1,1\rangle$ if and only if it may be written as 
$\langle a, x^{2}a\rangle$ where $a$ is a sum of two non-zero squares and $x$ is a non-zero integer. 
\end{corollary} 

\begin{proof} 
By the argument of Proposition \ref{Fermat}, the forms $\langle a, x^{2}a\rangle$ and $\langle 1,1\rangle$ 
are similar if and only if $a$ may be written as a sum of two squares. This occurs precisely when the square-free part of $a$ is not divisible by a prime which is $3 \mod 4$.
\end{proof} 

What then may be said of the other equivalence classes of quadratic forms? Suppose that $p$ and $q$ are distinct primes congruent to $3 \mod 4$. By the argument of Proposition \ref{Fermat}, the forms $\langle p,p\rangle$ and $\langle q,q\rangle$ are congruent if and only if there exists an integer solution to 
\[ p(\frac{a^{2} + b^{2}}{d^{2}}) = q\,,\] 
for integers $a,b,d$. Suppose first that $d$ is coprime to $p$, then $p(a^{2} + b^{2}) = qd^{2}$ is an integer equation, and $p$ divides $q$. 
However, this is absurd as $p$ and $q$ are distinct primes. Otherwise, $p^{2}$ divides $d^{2}$, say that $d = pt$. Then, one obtains an integer equation $a^{2} + b^{2} = pqt^{2}$. The square-free part of the right-hand side is divisible by a prime which is $3$ mod $4$, leading to a contradiction. Hence, the forms $\langle p,p\rangle$ for primes $p \equiv 3 \mod 4$ are all inequivalent. Similarly the form $\langle pq, pq\rangle$ is inequivalent to $\langle p,p\rangle$ and $\langle q,q\rangle$. Observe that a quadratic form of discriminant $1$ (i.e., square determinant) in dimension $2$ is necessarily of the form $\langle a,x^{2}a\rangle$ where $a$ and $x$ are integers. We may now classify forms in dimension $2$ of discriminant $1$ over $\mathbb{Q}$.

\begin{theorem} \label{dim2class}
Denote by $r_{3}(n)$ the product of all primes congruent to $3$ modulo $4$ which divide the square-free part of $n$. The quadratic forms $\langle n, x^{2}n\rangle$ and $\langle m, y^{2}m\rangle$ are similar if and only if $r_{3}(n) = r_{3}(m)$. 
\end{theorem} 

\begin{proof} 
Observe that 
\[ 
\begin{bmatrix} 1 & 0 \\ 0 & x^{-1} \end{bmatrix} 
\begin{bmatrix} n & 0 \\ 0 & x^{2}n \end{bmatrix} 
\begin{bmatrix} 1 & 0 \\ 0 & x^{-1} \end{bmatrix}  =  
\begin{bmatrix} n & 0 \\ 0 & n \end{bmatrix}\,, \] 
so nothing is lost by considering only the forms $\langle n,n\rangle$ and $\langle m,m\rangle$. 

Next, write $n = r_{3}(n) n'$ where $n'$ can be written as a sum of two integer squares, say $n' = a^{2} + b^{2}$. Then 
\[ 
\begin{bmatrix} a & b \\ -b & a \end{bmatrix} 
\begin{bmatrix} r_{3}(n) & 0 \\ 0 & r_{3}(n) \end{bmatrix} 
\begin{bmatrix} a & b \\ b & -a \end{bmatrix}  =  
\begin{bmatrix} n & 0 \\ 0 & n \end{bmatrix}\,,\] 
so $\langle n,n\rangle$ is similar to $\langle r_{3}(n), r_{3}(n)\rangle$. Since similarity in an equivalence relation, it follows that when $r_{3}(n) = r_{3}(m)$ the forms $\langle n,n\rangle$ and $\langle m,m\rangle$ are similar. 

Finally, suppose that $r_{3}(n) \neq r_{3}(m)$. Then without loss of generality there exists a prime $p \equiv 3 \mod 4$ such that $p \mid n$ and $p \nmid m$. If the forms were similar we would obtain a rational equation 
\[ (a^{2} + b^{2}x^{2})n = m\] 
Multiplying through by the square of the common denominator of $a$ and $b$ leaves an integer equation where the highest power of $p$ dividing the left hand side is odd. In contrast, the highest power of $p$ dividing the right hand side is even. This is a contradiction, and hence $\langle n,n\rangle$ and $\langle m,m\rangle$ cannot be similar. 
\end{proof} 

\subsection{Legendre and Hilbert symbols}

The standard treatments of quadratic forms adopt the language of Hilbert symbols for the discussion of quadratic forms in higher dimensions. Such symbols can be easily manipulated algebraically, while capturing the classification of Proposition \ref{dim2class} precisely. 

\begin{definition} 
For prime number $p$ and $a$ coprime to $p$, the \textit{Legendre symbol} $\left( \frac{a}{p} \right)$ is defined to be $1$ if $x^{2} \equiv a \mod p$ has a solution, and $-1$ otherwise. 
\end{definition} 

It is often convenient to set $\left( \frac{0}{p}\right) = 0$, though we will not require this convention. Many authors say that $a$ is a \textit{quadratic residue modulo} $p$ if $\left( \frac{a}{p} \right) = 1$, and a \textit{quadratic non-residue} otherwise. It follows directly from the definition of the Legendre symbol that $\left( \frac{a^{2}}{p} \right) = 1$ and that $\left( \frac{a}{p} \right) = \left( \frac{b}{p} \right)$ where $a \equiv b \mod p$. It is easily established that the Legendre symbol is multiplicative in the sense that $\left( \frac{ab}{p} \right) = \left( \frac{a}{p} \right) \left( \frac{b}{p} \right)$ which reduces the evaluation to prime arguments. Gauss' celebrated law of quadratic reciprocity gives an efficient reduction for evaluation of such symbols.

\begin{theorem}[Gauss, cf. Chapter 5, \cite{IrelandRosen}]
Let $p$ and $q$ be odd primes. Then 
\[ \left( \frac{p}{q} \right)\left( \frac{q}{p} \right) = (-1)^{\frac{p-1}{2}{\frac{q-1}{2}}}\,.\]
The symbol $\left( \frac{-1}{q} \right)$ evaluates to $1$ if $q \equiv 1 \mod 4$ and $-1$ if $q\equiv 3 \mod 4$. 
The symbol $ \left( \frac{2}{q} \right)$ evaluates to $1$ if $q \equiv \pm 1 \mod 8$ and $-1$ if $q \equiv \pm 3 \mod 8$. 
\end{theorem} 
For the reader unfamiliar with Legendre symbols, we provide sample computations below.
\begin{example}
The symbol $\left( \frac{4}{p} \right)$ evaluates to $1$ for every prime since $4$ is a square in the integers. 
The symbol $\left( \frac{2}{7} \right)$ evaluates to $1$ because $3^{2} \equiv 2 \mod 7$. The symbol 
 $\left( \frac{2}{11} \right) = -1$ because $2$ is not a quadratic residue. (This may be verified exhaustively. 
 Alternatively, observe that $2 \equiv -9 \mod 11$, then use that the negative of a residue is a non-residue when $p \equiv 3 \mod 4$.)

Larger examples are evaluated by repeatedly flipping the terms in the symbol, factoring and evaluating `easy' terms. 

\[  \left( \frac{31}{103} \right) = (-1)  \left( \frac{103}{31} \right) =  (-1)\left( \frac{10}{31} \right) = \]
\[(-1)\left( \frac{2}{31} \right) \left( \frac{5}{31} \right) =  (-1)\left( \frac{2}{31} \right) \left( \frac{31}{5} \right) = (-1)\left( \frac{2}{31} \right) \left( \frac{1}{5} \right) = -1\]   

In the last step, we used that the remaining Legendre symbols evaluate to $1$. (Note that $8^{2} \equiv 2 \mod 31$.)
\[ \left( \frac{29}{151} \right) = \left( \frac{151}{29} \right) = \left( \frac{6}{29} \right) = \left( \frac{2}{29} \right)\left( \frac{3}{29} \right) = \left( \frac{2}{29} \right)\left( \frac{29}{3} \right) = \left( \frac{2}{29} \right)\left( \frac{2}{3} \right) = 1\]
We can verify with a computer that $28^{2} \equiv 29 \mod 151$ so that the Legendre symbol is indeed correct. 
\end{example} 

Next, we introduce the Hilbert symbol, which was developed precisely to characterise equivalence of quadratic forms. 

\begin{definition} 
Define the \textit{Hilbert symbol} $(a,b)_{p}$ for non-zero integers $a,b$ and prime $p$ to be $1$ if the equation 
\[ ax^{2} + by^{2} = z^2 \] 
has a non-trivial solution in the $p$-adic numbers $\mathbb{Q}_{p}$ and $-1$ otherwise.  
\end{definition} 

The definition mentions the $p$-adic numbers, which properly form the \textit{local} part of the \textit{local-global} 
theory of quadratic forms. They will not be necessary for our purposes: the value of the Hilbert symbol can 
always be determined in terms of Legendre symbols. Note that the integers form a subring of $\mathbb{Q}_{p}$ 
for all primes $p$, so an integer solution of Hilbert's equation means that the corresponding Hilbert symbols 
evaluate to $1$ for any prime $p$. The following identities for the Hilbert symbol are immediate from the definition: 
\[ (a, b)_{p} = (b, a)_{p}, \,\,\, (a^{2}, b)_{p} = 1, \,\,\, (a,-a)_p=1, \,\,\, (a,(1-a))_p=1\] 
the first identity by swapping variables $x$ and $y$, the second by setting $x = a^{-1}$ and $y = 0, z=1$, for the third $x=y, z=0$ and for the fourth $x=y=z$. 

A proof for the following useful identity may be found in Chapter 3 of \cite{Serre}  
\[ (a,b)_p = -1^{\alpha \beta \epsilon(p)} \left( \frac{u}{p}\right)^\beta \left( \frac{v}{p}\right)^\alpha, \qquad \text{for any odd prime } p,\]
where $a=p^\alpha u, b=p^\beta v$ with $\gcd(uv,p) = 1$ and $\epsilon(p)=\displaystyle\frac{p-1}{2}$. 

The following properties of the Hilbert symbol may be deduced from the displayed equation, and are worth stating directly: 
\begin{enumerate} 
\item $(a_{1}a_{2}, b)_{p} = (a_{1}, b)_{p}(a_{2}, b)_{p}$, the Hilbert symbol is \textit{bilinear}.
\item $(a,b)_p=(a,-ab)_p=(a,(1-a)b)_p$. 
\item If $a$ and $b$ are both coprime to $p$ then $(a,b)_{p} = 1$. 
\item $(a, p)_{p} = \left( \frac{a}{p} \right)$, this case of the Hilbert symbol reduces to the Legendre symbol.
\item $(p,p)_{p} = \left(\frac{-1}{p}\right)$ so $(p,p)_{p} = 1$ if $p \equiv 1 \mod 4$ and $(p,p)_{p} = -1$ if $p \equiv 3 \mod 4$. 
\end{enumerate} 

\begin{example} 
To illustrate the evaluation of the Hilbert symbol, we compute $(21, 33)_{3}$. 
By bilinearity, this the symbol splits into prime factors. By the remaining properties of the Hilbert symbol, 
each term is either trivial or evaluated in terms of a Legendre symbol as follows: 
\[ (3,3)_{3}(3, 11)_{3}(7, 3)_{3}(7, 11)_{3} =  \left( \frac{-1}{3} \right) \left( \frac{11}{3} \right) \left( \frac{7}{3} \right) = 1\,.\] 
\end{example}

\begin{remark} 
In this paper, we restrict attention almost entirely to odd primes. 
The rules for manipulating $(a, b)_{2}$ are slightly more complicated than 
those for odd primes, but are given in any specialist text on quadratic forms. 
We also overlook the so-called infinite prime (called the prime $-1$ by Conway) 
which relates to the solvability of the equation $ax^{2} + by^{2} = 1$ in the real numbers, 
so $(a, b)_{\infty} = 1$ provided at least one of $a$ and $b$ is positive. Since we deal only 
with positive definite matrices in this paper, the Hilbert symbol at infinity is always $1$.

Artin's \textit{Global Product Formula} for Hilbert Symbols states that all-but-one of the Hilbert symbols 
determines the last one, where we quantify over primes. Hence, if a (positive definite) quadratic form differs from 
$I_{n}$ at $p = 2$, it also differs at an odd prime, and this fact can be detected there. So computing 
Hilbert symbols at odd primes suffices for our purposes. 
\end{remark} 

Let us conclude this section with an explicit demonstration that Hilbert symbols 
characterise the equivalence of quadratic forms in two dimensions.

\begin{theorem}[cf. Theorem \ref{dim2class}] \label{HMdim2}
Let $m,n,x,y$ be positive integers. Quadratic forms $\langle n,nx^{2}\rangle$ and 
$\langle m,my^{2}\rangle$ are similar if and only if 
$(n,nx^{2})_{p} = (m,mx^{2})_{p}$ for every odd prime.
\end{theorem} 

\begin{proof} 
First, by elementary properties of the Hilbert symbol 
\[ (n,nx^{2})_{p} = (n,n)_{p} (n, x^{2})_{p} = (n,n)_{p} \,.\] 
Write $n = a^{2}n' p^{t}$ where $n'$ is square-free and coprime to $p$, and $t \in \{0,1\}$ (note that $p$ may divide $a$). 
Then $(n,n)_{p} = (p^{t}, p^{t})_{p}$. This symbol evaluates to $-1$ if and only if $t = 1$ and $p \equiv 3 \mod 4$. 
So, the Hilbert symbol detects the primes congruent to $3$ modulo $4$ which divide the square-free 
part of $n$. The proof of Theorem \ref{dim2class} provides an explicit demonstration that the quadratic forms are 
similar when the Hilbert symbols agree.
\end{proof} 

Observe that Theorem \ref{HMdim2} is precisely equivalent to Theorem \ref{dim2class}; the rules for manipulating the Hilbert symbol allow for some simplifications of the proof. Properties of the Hilbert symbol will be essential to compute invariants in higher dimensions.

Let us conclude this section with a summary of the discussion up to this point: 
\begin{enumerate} 
\item In dimension $2$, a quadratic form of discriminant $1$ is necessarily of the form $\langle a, x^{2}a\rangle$. 
Every such form is equivalent to $\langle m,m\rangle$, for square-free integer $m$.
\item Two such forms are equivalent if and only if the set of primes congruent to $3 \mod 4$ dividing the square-free part of $m$ are equal. We established this via elementary arguments, and the similarity matrix may be computed explicitly, provided all relevant primes can be written as a sum of two squares. 
\item The Hilbert symbols associated with $\langle m,m\rangle$ are $(m,m)_{p}$, where we allow $p$ to range over the odd primes. The bilinearity of the Hilbert symbol reduces its evaluation to the Legendre symbol, and Gauss' Reciprocity Law allows practical evaluation of the Legendre symbol. The Hilbert symbol $(m,m)_{p}$ evaluates to $-1$ if and only if $p$ is a prime congruent to $3 \mod 4$ dividing the square free part of $m$.
\item In dimension two, positive definite quadratic forms are equivalent if and only if their Hilbert symbols agree at all odd primes.\end{enumerate} 

In the next section, we extend the Hilbert symbol to an invariant of quadratic forms in dimension $n$.

\section{Invariants for Quadratic forms in dimension $n$}

Recall that a rational quadratic form in $n$ dimensions is represented by a symmetric $n \times n$ matrix $A$ with integer entries (with respect to a fixed basis). With respect to another basis, the matrix of the form is $M^{\top}AM$ where $M$ is the relevant change-of-basis matrix. 

\begin{definition}
A function $f: \mathrm{Mat}_{n}(k) \rightarrow k$ is an \textit{invariant of quadratic forms} if $f(A) = f(M^{\top}AM)$ for any invertible matrix $M$ and any symmetric matrix $A$. 
\end{definition}

We begin with two obvious invariants of quadratic forms. 

\begin{proposition} 
The \textit{discriminant} of a square integer matrix $A$ is the square-free part of the determinant of $A$. 
The \textit{signature} of $A$ is the number of positive, zero and negative eigenvalues of $A$. 
If $A$ is a matrix representing the quadratic form $Q$, then the discriminant and signature of $A$ are invariants of $Q$.
\end{proposition} 

It is trivial to see that the discriminant is an invariant of quadratic forms. That the signature is an invariant does 
require proof: in the special case of the real field, this is Sylvester's \textit{Law of Inertia}. The proof is contained in any text discussion quadratic forms, \cite{Serre, OMeara}. In this section, we develop more subtle arithmetic 
invariants of quadratic forms. This theory was developed by Hilbert and Minkowski in the early twentieth 
century, and placed in its final form by Hasse. We follow the exposition of B. W. Jones quite closely, \cite{JonesB}. 

Recall that a \textit{minor} of a matrix is the determinant of a square submatrix. The \textit{first minor} $M_{i,j}$ of an $n \times n$ matrix $M$ is obtained by deleting row $i$ and column $j$, and taking the determinant of the $(n-1)\times (n-1)$ submatrix remaining. For $1 \leq k \leq n$, the $k^{\textrm{  th}}$ \textit{leading minor} of $M$ is the determinant of the $k \times k$ submatrix in the upper left of $M$, which we denote $m_{k}$. In particular, $m_{n-1} = M_{n,n}$ and $m_{n} = \det(M)$.
 
\begin{definition} 
Let $A$ be an $n \times n$ symmetric matrix with rational entries. The \textit{Pall invariant} of $A$ at the prime $p$ is 
\[ c(A, p) = (-1, -m_{n})_{p} \prod_{i=1}^{n-1} (m_{i}, -m_{i+1})_{p}\,,\]
where $m_{i}$ is the $i^{\textrm{th}}$ leading minor of $A$. 
\end{definition} 

To prove that $c(A, p)$ is an invariant of quadratic forms we require the following lemma on determinants, which was already well-known in the nineteenth century. 

\begin{lemma}\label{Hadamard-Minors-Lemma}  
Let $M$ be an $n\times n$ symmetric positive definite matrix. 
Denote by $M_{i,j}$ a first minor of $M$, and $i\neq j$, let $M_{[i,j]}$ be the $(n-2)\times (n-2)$ minor obtained by removing the $i$-th and $j$-th rows and the $i$-th and $j$-th columns of $M$. Then
\[\det(M)\det(M_{[i,j]})=M_{i,i}M_{j,j}-(M_{i,j})^2,\]
\end{lemma}
\begin{proof}
Since $M$ is positive definite, it is invertible. Write $N$ for the inverse of $M$. Decompose both matrices into block matrices, 
\[M=\begin{bmatrix}
M_1 & M_2\\
M_3 & M_4
\end{bmatrix}\,, \,\,\, N=\begin{bmatrix}
N_1 & N_2\\
N_3 & N_4
\end{bmatrix}\,,\]
where $M_{1}$ and $N_{1}$ are $k \times k$ and $M_{4}$ and $N_{4}$ are $(n-k)\times (n-k)$. 
Consider the matrix identity
\[\begin{bmatrix}
N_1 & N_2\\
N_3 & N_4
\end{bmatrix}
\begin{bmatrix}
M_1 & 0\\
M_3 & I
\end{bmatrix}
=
\begin{bmatrix}
I & N_2\\
0 & N_4
\end{bmatrix}\,.
\]
Taking determinants, $\det(N)\det(M_1)=\det(N_4)$. 
Since the determinant of $M$ is unchanged after a \textit{symmetric} row/column permutation, we may assume without loss of generality that $M_1=M_{[i,j]}$ and that 
\[M_4=\begin{bmatrix}
m_{ii} & m_{ij}\\
m_{ij} & m_{jj}
\end{bmatrix}.\]
Since $N$ is the inverse of $M$, the entries of $N$ are first minors of $M$, multiplied by $\det(M)^{-1}$. 
So, up to a $-1$ factor which cancels out in the final formula, we find that 
\[N_4=\frac{1}{\det(M)}\begin{bmatrix}
M_{i,i} & -M_{i,j}\\
-M_{i,j} & M_{j,j}
\end{bmatrix}.\]
Therefore
\[\det(N)\det(M_{1}) = \det(M)^{-1} \det(M_{[i,j]})=\det(N_4)=\det(M)^{-2}(M_{i,i}M_{j,j}-(M_{i,j})^2).\]
Multiplying by $\det(M)^2$ concludes the proof.\qedhere
\end{proof}

By the theory of the Row Echelon Form, any invertible matrix may be reduced to the identity by a sequence of elementary row operations: permuting rows, adding one row to another and multiplying a row by a scalar. Consequently, $\GL_{n}(\mathbb{Q})$ is generated by the matrices associated with these row operations.

\begin{theorem}[Section 11, \cite{JonesB}] \label{Pall}
Let $M$ be a positive definite rational matrix. For each $N\in\GL_n(\mathbb{Q})$, 
\[c(M,p)=c(N^{\intercal}MN, p)\,.\]
\end{theorem}
\begin{proof}
It suffices to show that the Pall invariants are preserved by matrices representing elementary row operations.\\
To show that the Pall invariants are preserved by arbitrary (simultaneous) permutation of rows and columns, it will suffice to show that they are preserved by transpositions $(i, i+1)$ of adjacent rows or columns, since these permutations generate the symmetric group. 

The leading minors of the matrix $M'=P^{\intercal}MP$ coincide with those of $M$ except possibly for 
$m_i$ and $m'_i$, since all other leading minors are either unchanged, or have a pair of rows and columns swapped, 
leaving the determinant unchanged. It will suffice to show that
\[(m_{i-1},-m'_i)_{p}(m'_i,-m_{i+1})_{p}=(m_{i-1},-m_i)_{p}(m_i,-m_{i+1})_{p}.\]
Multiplying by $(m_{i-1},-1)$ on both sides and using bilinearity of the Hilbert symbol, we find that the above equation is equivalent to
\[(m'_i,-m_{i-1}m_{i+1})_{p}=(m_i,-m_{i-1}m_{i+1})_{p}.\]
Now, applying Lemma \ref{Hadamard-Minors-Lemma} to the leading $(i+1)\times (i+1)$ submatrix underlying the minor $m_{i+1}$ with indices $i$ and $j=i+1$, 
\[m_{i+1}m_{i-1}=m'_im_i-d^2,\]
for some $d\in \mathbb{Q}$. Thus we must check that
\[(m'_i,d^2-m'_im_i)_{p}=(m_i,d^2-m'_im_i)_{p}\,,\]
or equivalently $(m'_im_i,d^2-m'_im_i)_{p}=1$. If $d = 0$ the result holds since $(x,-x)_{p} = 1$ for any non-zero $x$ and any prime $p$. Otherwise, 
\[ m'_{i} m_{i} X^{2} + (d^{2} - m'_{i}m_{i}) Y^{2} = 1 \] 
has a solution by taking $X = Y = d^{-1}$ and so the Hilbert symbol evaluates to $1$ as required. 
Unravelling this chain of equivalences to its start, we conclude that the Pall invariants are preserved by 
simultaneous permutation of rows and columns.

Next, consider the elementary row operation which adds row $i$ to row $j$, and adds column $i$ to column $j$. 
Since rows may be permuted arbitrarily, we assume without loss of generality that $i<j$. This implies that all leading principal minors are unchanged, so the Pall invariant stays constant.

Finally, in the case of scalar multiplication, column and row $i$ are both multiplied by a non-zero scalar $\lambda$. The involved determinants are multiplied by $\lambda^2$. 
Again working with the definition of the Hilbert symbol, it is clear that  $(\lambda^2 m_{j}, b)_{p}=(m_{j},b)_{p}$ 
for any $b\in \mathbb{Q}^{\ast}$. Hence the Pall invariants are preserved by scalar multiplication. 

This concludes the proof: any invertible matrix $N$ may be expressed as a product of elementary row operation matrices, and the elementary row operations preserve the Pall invariants. So the Pall invariants of $N^{\top}MN$ and $M$ agree. 
\end{proof}

Since the Hilbert symbol is defined only at non-zero arguments, we used implicitly in the argument above that the 
minors of $M$ are non-vanishing. This follows from the assumption that $M$ is positive definite via Sylvester's criterion (but can be evaded by slightly lengthier \textit{ad hoc} arguments in the general case).

\begin{remark} 
The invariants of Theorem \ref{Pall} are a complete set, in the sense that two quadratic forms having the 
same rank, discriminant, and inertia, and taking the same value at all primes (including now $p = 2, \infty$) are necessarily similar over the rationals. This is essentially the statement of the Hasse-Minkowski theorem. Thus, the (difficult) question of rational equivalence is reduced to (easy) questions about much larger local fields, which are the reals and the $p$-adics, where equivalence testing is reduced to Pall invariants and Hilbert symbols. The Hasse-Minkowski theorem is often referred to as the \textit{local-global} principle. For a complete proof, the reader is referred to Serre's \textit{Arithmetic}, \cite{Serre}.
\end{remark} 

The astute reader will notice that we justified the reduction of quadratic forms to diagonal matrices by applying a change of basis operation in Section \ref{Background}, while in Theorem \ref{Pall} we applied a `quadratic' transformation to the quadratic space by replacing the matrix $M$ of the quadratic form by the matrix $N^{\top}MN$. This is precisely the distinction between computing the image of a linear transformation under a change of basis and considering a pair of distinct but conjugate linear transformations.

In any case, Theorem \ref{Pall} allows us to reduce to diagonal matrices, at which point the invariants may be expressed more concisely. 

\begin{definition}
The \textit{Hasse-Minkowski invariant} of a polarised (i.e. diagonal) quadratic form $Q = \langle a_{1}, \ldots, a_{n} \rangle$ at the prime $p$ is 
\[ H(Q, p) = \prod_{i < j} (a_{i}, a_{j})_{p} \,.\] 
\end{definition} 

\begin{proposition} 
At any odd prime, the Hasse-Minkowski and Pall invariants are equal for a polarised (diagonal) form of discriminant $1$.
\end{proposition} 

\begin{proof} 
Let $Q$ be a quadratic form. By Theorem \ref{Pall}, the Pall invariants do not depend on the choice of symmetric matrix used 
to represent the form. By Proposition \ref{RowOps}, we may take $Q = \langle a_{1}, \ldots, a_{n}\rangle$. We will reduce the Pall invariant to the Hasse-Minkowski invariant. 

Recall that the Hilbert symbol satisfies the following identities: $(a, -a)_{p} = 1$ and $(a, bc)_{p} = (a, b)_{p}(a,c)_{p}$, and that the $k^{\textrm{th}}$ minor is defined as $m_{i} = \prod_{i=1}^{k}a_{i}$. Then 
\[ (m_{k}, -m_{k+1})_{p} = (m_{k}, -m_{k})_{p} (m_{k}, a_{k+1})_{p} = \prod_{i=1}^{k} (a_{i}, a_{k+1})_{p}\,.\] 
By hypothesis, the discriminant is $1$ which means that \[\det(Q)=m_n=a^2 \text{ and } (-1,-m_n)_p=(-1,-a^2)_p=(-1, -1)_{p} = 1\] since $-1$ is coprime to $p$. So for any odd prime, the Pall symbol evaluates to the Hasse-Minkowski symbol as required:
\[ c(Q, p) = (-1, -1)_{p} \prod_{k=1}^{n} \prod_{i=1}^{k-1} (a_{i}, a_{k})_{p} = \prod_{i<k} (a_{i}, a_{k})_{p} = H(Q, p)\,. \qedhere\] 
\end{proof} 

\begin{remark} 
We caution the reader that, as defined, the Pall Invariant of $I_{n}$ is $-1$ at $p = 2, \infty$, while the Hasse-Minkowski invariants of the identity matrix are all $1$. Various authors have adopted different conventions for classifying quadratic forms: arguably the most natural quadratic form is the one composed of a direct sum of hyperbolic planes (over $\mathbb{Q}$, a hyperbolic plane has matrix $\langle 1, -1\rangle$), and these have all Pall invariants equal to $1$. 
\end{remark} 

\begin{example} 
To illustrate the computation of Hasse-Minkowski invariants (and Hilbert symbols), let us decide whether the form $\langle 1, 2, 7, 14\rangle$ is rationally equivalent to $\langle 1,1 , 1, 1\rangle$. 
Both forms are positive definite (and so have the same signature) and have discriminant $1$.
For this it suffices to consider the local invariant at $7$, which is 
\[ H(Q, 7) = (1, 2)_{7} (1, 7)_{7} (1, 14)_{7} (2, 7)_{7} (2, 14)_{7} (7, 14)_{7} \] 
Using bilinearity, we expand the composite terms: 
\[ H(Q, 7) = (1, 2)_{7}^{2} (1, 7)_{7}^{2} (2, 2)_{7} (2, 7)_{7}^{3} (7, 7)_{7} \] 
Now, we cancel square terms: 
\[ H(Q, 7) =  (2, 2)_{7} (2, 7)_{7} (7, 7)_{7} \] 
Of these, the first is $1$ because both arguments are coprime to $7$, and the second is likewise $1$ because $2$ is a square mod $7$. But $7 \equiv 3 \mod 4$, so the last term is $-1$ and the local invariant differs from that of the standard form. Hence the forms are inequivalent over $\mathbb{Q}$. The local invariants typically do not provide any clue about the rational matrix relating one form to the other. 
\end{example} 

At last, we give the long-promised sufficient criterion to exclude a positive definite matrix from being a Gram matrix. 

\begin{theorem} \label{mainGramthm}
Let $G$ be a positive definite rational matrix. 
Suppose that at least one of the following holds:
\begin{enumerate} 
\item $G$ has discriminant different to $1$.
\item For some odd prime, the Hasse-Minkowski symbol of $G$ evaluates to $-1$. 
\end{enumerate} 
Then $G$ is not a Gram matrix; there is no rational matrix $M$ such that $M^{\top}M = G$.
\end{theorem} 

\section{Applications in Design theory} 

We are now in a position to give proofs of a number of results in design theory, which depend essentially 
on computing the Hasse-Minkowski invariants of certain potential Gram matrices. Recall that these are only 
necessary (and not sufficient) conditions for the existence of designs. Tracing the historical development of these 
results, we begin with the Bruck-Ryser theorem on projective planes. 

\subsection{The Bruck-Ryser Theorem} 

Recall that a projective plane of order $n$ is a finite geometry in which each line contains $n+1$ points and 
every pair of lines meet at a unique point. This definition is equivalent to a ${0,1}$-matrix of order $n^{2} + n + 1$ which satisfies the equation, 
\[ M^{\top}M = nI + J\] 
where $J$ is the all-ones matrix. To prove the non-existence of certain projective planes, 
it would be sufficient to show that the quadratic 
forms of $I_{n}$ and $nI_{n} +J$ differ. We will do this by computing the Hasse-Minkowski 
invariants of both forms. We begin by \textit{polarising} the matrix $nI + J$. While every symmetric matrix is 
diagonalisable over the reals by an orthogonal matrix (as a consequence of the spectral theorem), such a 
matrix does typically not have entries over the rationals. As such, we need to be a little more careful - while 
we labour the point a little here, it indicates some of the techniques used in working with quadratic forms. 
 
\begin{proposition}\label{polarBRC}
The $d \times d$ matrices $nI+ J$ and $\langle (n+d)d, (2\cdot 1)n, (3\cdot 2)n \ldots, (d\cdot d-1)n\rangle$ are congruent. 
\end{proposition}

\begin{proof} 
Since every vector is an eigenvector of $nI$, it suffices to choose an orthogonal eigenbasis for $J$ in which 
each basis vector has rational entries. This may be accomplished as follows: 
\[ f_{1}^{\top} = (1,1,1,\ldots, 1), \,\,\, f_{i}^{\top} = (1,1,\ldots,1, -i+1, 0, \ldots, 0), \,\,\, 2 \leq i \leq d\] 
where $f_{i}^{\top}$ contains $-i+1$ in co-ordinate $i$, with $1$'s to the left and $0$'s to the right. 
By linearity, $(n I + J)f_{1} = (n+d)f_{1}$ and $(nI +J)f_{i} = nf_{i}$ for $2 \leq i \leq d$. 
Let $F$ be the matrix with $f_{i}$ in the $i^{\textrm{th}}$ column. 
Then $D = F^{\top} (nI + J) F$ is diagonal, with $D_{1} = (n+d)d$ and $D_{i} = i(i-1)n$ for $2 \leq i \leq d$. 
\end{proof} 

When $d = n^{2} + n + 1$, the eigenvalues of $nI + J$ are $(n+1)^{2}$ with multiplicity $1$ and $n$ with multiplicity $n^{2} + n$. 
Proposition \ref{polarBRC} rewrites $nI +J$ as a product of the usual (real) diagonalisation of $nI+ J$ with 
a matrix $\langle d, 1\cdot 2, 2\cdot 3, \ldots, (n^{2} + n)\cdot(n^{2} + n + 1)\rangle$. It will be convenient for us 
to eliminate this second diagonal matrix from later computations. Let us compare the local invariants 
of a matrix product to those of its terms. 
 
\begin{proposition} \label{prop:split}
Let $A = \langle a_{1}, \ldots, a_{n}\rangle$ and $B = \langle b_{1}, \ldots, b_{n}\rangle$ be quadratic forms. 
Then for any prime $p$, the local invariant of $AB$ at $p$ is equal to 
\[ H(AB, p) = H(A, p)H(B, p) (\Delta A, \Delta B)_{p} \prod_{i=1}^{n} (a_{i}, b_{i})_{p}\,, \] 
where $\Delta A$ is the discriminant of $A$.
\end{proposition} 

\begin{proof} 
By bilinearity of the Hilbert symbol, 
\[ H(AB, p)= \prod_{i < j} (a_{i}, a_{j})_{p} (a_{i}, b_{j})_{p} (a_{j}, b_{i})_{p} (b_{i}, b_{j})_{p} \] 
and since the terms commute, 
\[  H(AB, p) = \prod_{i <  j} (a_{i}, a_{j})_{p} (b_{i}, b_{j})_{p} \prod_{i\neq j}(a_{i}, b_{j})_{p} \,. \] 
Now, add the diagonal terms, $\prod_{i} (a_{i}, b_{i})^{2}_{p} = 1$, gather one copy of each term into the rightmost product of above, 
\[  H(AB, p)= \prod_{i < j} (a_{i}, a_{j})_{p} (b_{i}, b_{j})_{p} \prod_{j}( \prod_{i} a_{i}, b_{j})_{p} \prod_{i} (a_{i}, b_{i})_{p} \,. \] 
Finally, by bilinearity in the second argument, 
\[ H(AB, p) = H(A, p)H(B, p) (\Delta A, \Delta B)_p \prod_{i} (a_{i}, b_{i})_{p} \,. \qedhere\] 
\end{proof} 

\begin{theorem}\label{thm:polarising}
Let $d = n^{2} + n + 1$, for some positive integer $n$.  Then the quadratic forms $nI_{d} + J_{d}$ and the polarised 
form $\langle (n+1)^{2}, n, n, \ldots, n\rangle$ are congruent. 
\end{theorem} 

\begin{proof} 
Set $a_{1} = (n+1)^{2}$ and $a_{i} = n$ for $2 \leq i \leq d$, so that the $a_{i}$ are the eigenvalues of $nI + J$; and set $b_{1} = d$ and $b_{i} = i(i-1)$. We will apply Proposition \ref{prop:split} to show that the local invariant $(AB, p)$ is equal to $(A, p)(B, p)$. First, $\det(A) = (n+1)^{2}n^{n^{2}+n}$ is a square, so that $\Delta A = 1$. Hence $(\Delta A, \Delta B)_p = 1$. Similarly, $(a_{1}, b_{1})_{p} = ((n+1)^{2}, b_{1})_{p} = 1$ since the Hilbert symbol is $1$ when either argument is a square of an integer. The following product telescopes: 
\[ \prod_{i=2}^d (a_i,b_i)_p = \prod_{i=2}^{d} (n, i (i-1))_{p} = \prod_{i=2}^{d} (n, i-1)_{p} (n, i)_{p} = (n, 1)_{p}(n,d)_{p} \,.\] 
The first term evaluates to $1$. For the second, observe that it is trivially $1$ for any prime not dividing $n$ or $d$. If $p$ divides $n$, then $d = n^{2} + n + 1 \equiv 1 \mod p$ and the term is $1$. If $p$ divides $d$ then $(n+1)^{2} = n^{2} + 2n + 1 \equiv n \mod d$ because $d = n^{2} + n + 1$, and the term is $1$ since $n$ is a square $\bmod$ $p$. Hence the product of diagonal terms vanishes always. 

Finally, we evaluate the local invariants at the form $B = \langle d, 2\cdot 1, 3 \cdot 2, \ldots, (m-1)m \rangle$. The local invariant at $p$ is 
\begin{multline} \prod_{i=2}^{d} (d, i(i-1))_p \prod_{2 \leq i < j\leq d} (i(i-1), j(j-1))_p=\\ \prod_{i=2}^d (d, i)_p (d, i-1)_p \prod_{2 \leq i<j \leq d} (i,j)_p(i,j-1)_p(i-1, j)_p(i-1, j-1)_p \end{multline} 
The first product telescopes, leaving only the terms $(d,1)(d,d)$. For a fixed value of $i$, the terms in the second product telescope to leave a remainder $(i,i)_p(i,d)_p(i-1,i)_p(i-1, d)_p$. In the product
\[ \prod_{2 \leq i < d}(i,i)_p(i,d)_p(i-1,i)_p(i-1, d)_p\,,\]
the terms $(i,d)_p$ and $(i-1,d)_p$ cancel out except for $(d-1,d)_p$ and $(1,d)_p$. The terms $(d,1)_p$ and $(1,d)_p$ cancel by symmetry.
Collecting the remaining terms we can see that the local invariant is equal to 
\[ \prod_{2 \leq i \leq d}(i,i)_p(i-1,i)_p. \]
Using equalities for the Hilbert symbols we can see that $(i,i)_p(i-1,i)_p=(i,(1-i)i)_p(i-1,(1-i)i)_p=((i-1)i,(1-i)i)_p=1$.
In fact, this is not surprising, since this matrix is given by $FF^{\top}$. 

Hence the local invariants of $F^{\top}(nI + J)F$ agree with those of the diagonal form $\langle (n+1)^{2}, n, \ldots, n \rangle$ as required.
\end{proof} 

Finally, we prove the Bruck-Ryser theorem. 

\begin{theorem}[Bruck-Ryser]\label{BruckRyser}
Suppose that $\Pi$ is a projective plane of order $n$ where $n \equiv 1,2 \mod 4$. Then $n = a^{2} + b^{2}$, for integers $a$ and $b$.
\end{theorem} 

\begin{proof} 
By Theorem \ref{thm:polarising}, it suffices to compute the invariants of the quadratic form $A = \langle (n+1)^{2}, n, \ldots, n\rangle$ where there are $n^{2}+ n$ terms equal to $n$, and compare these to the invariants of the identity matrix. A single local invariant equal to $-1$ proves non-existence of the corresponding projective plane, while having all local invariants equal to $1$ is inconclusive. The Hasse-Minkowski Invariant at a prime $p$ is 
\[ H(A, p) = (n^{2} +2n+1, n)_{p}^{n^{2} + n} (n, n)_{p}^{\binom{n^{2}+n}{2}}\,. \]
Clearly the first term vanishes, while the exponent of the second term is odd precisely when $n \equiv 1,2 \mod 4$. 

Suppose now that $p$ is an odd prime divisor of the square-free part of $n$. The local invariant at $p$ reduces to the condition $(n,n)_{p} = 1$. By the definition of the Hilbert symbol, this is equal to $(-1, p)_{p} = (n, -1)_{p}$, which is $1$ if and only if $nx^{2} - y^{2} = z^{2}$ has a solution. But this is precisely equivalent to $nx^{2} = z^{2} + y^{2}$, which requires that $n$ is a sum of two squares. \qedhere
\end{proof} 

Recall that Fermat's Theorem on sums of two squares gives a characterisation of the permissible values of $n$ in Theorem \ref{BruckRyser}: they are precisely those for which the square free part of $n$ is not divisible by a prime congruent to $3$ modulo $4$. Thus, Theorem \ref{BruckRyser} rules out projective planes of order $6, 14, 21, 22, 30, \ldots$

\subsection{Bruck-Ryser-Chowla theorem}

An extension of the Bruck-Ryser theorem to arbitrary symmetric designs was achieved in collaboration with Chowla. Numerous proofs appear in the literature; we sketch an extension of the arguments given for projective planes. 

\begin{proposition} \label{prop:combmat}
The local invariants of the form represented by $(a-b)I_{d} + bJ_{d}$ are given by 
\[(a-b+db, a-b)_{p}^{d-1}(a-b, a-b)_{p}^{\binom{d-1}{2}}(a-b, d)_{p}(a-b+db, d)_{p} \,.\]
\end{proposition} 

\begin{proof} 
Let $n = \frac{a-b}{b}$, and consider $b( nI_{d} + J_{d})$. Since $bI_{d}$ is scalar, it commutes with row and column operations. Apply Proposition \ref{polarBRC} to find that the given matrix is congruent to the diagonal matrix $\langle (a-b+bd)d, 2\cdot1\cdot(a-b), \ldots, d(d-1)(a-b) \rangle$. 

Set $A = \langle a-b-bd, a-b, \ldots, a-b\rangle$ and $B = \langle d, 2\cdot 1, \ldots d(d-1)\rangle$ and apply Lemma \ref{prop:split}. The argument given in Theorem \ref{thm:polarising} holds for $B$ and shows that all its local invariants are $1$. Also as in that proof, the product $\prod_{i=1}^{d} (a_{i}, b_{i})_{p}$ telescopes. After some computation, the reader finds that 
\[ H(AB, p) =(a-b+db, a-b)_{p}^{d-1}(a-b, a-b)_{p}^{\binom{d-1}{2}}(a-b, d)_{p}(a-b+db, d)_{p}\,,\]
as required. 
\end{proof}

From this follows the full Bruck-Ryser-Chowla theorem. 

\begin{theorem} \label{BRC}
Suppose that $D$ is the incidence matrix of a symmetric design with parameters $(v,k,\lambda)$. 
If $v$ is even, then $k-\lambda$ is the square of an integer. If $v$ is odd then for all primes $p$ the 
Hilbert symbol $(k-\lambda, (-1)^{v-1/2}\lambda)_{p} = 1$. 
\end{theorem} 

\begin{proof} 
Suppose first that $v$ is even. Then $\det(DD^{\top}) = (k + (v-1)\lambda)(k-\lambda)^{v-1} = k^{2}(k-\lambda)^{v-1}$ which must be a perfect square. Hence $k-\lambda$ is necessarily a square. (No condition arises in this way when $v$ is odd.)

Recall that $k(k-1) = (v-1)\lambda$ for any symmetric design, so in particular $(v-1)\lambda + k$ is a perfect square. Apply Proposition \ref{prop:combmat} with $a = k$, $b = \lambda$ and $d = v$. 
Observe that $a - b + bd = k^{2}$, and as usual in design theory write $n = k - \lambda$ so that the 
local invariants of $nI_{v} + \lambda J_{v}$ now take the shape 
\[ (n,n)_{p}^{\binom{v-1}{2}}(n, v)_{p} \]
As in the proof of Theorem \ref{Pall}, observe that $(n, k^{2}-n)_{p} = 1$ for all primes $p$, by solving 
$nX^{2} + (k^{2}-n)Y^{2} = 1$ explicitly with $X = Y = k^{-1}$. But $k^{2}-n = v\lambda$, and bilinearity 
of the Hilbert symbol gives $(n, v)_{p}(n, \lambda)_{p} = 1$ or equivalently $(n,v)_{p} = (n, \lambda)_{p}$. 
Additionally, $\binom{v-1}{2}$ is even if $v \equiv 1 \mod 4$ and odd if $v \equiv 3 \mod 4$. So the local invariants 
coincide with the expression 
\[ (n, (-1)^{v-1/2})_{p}(n, \lambda)_{p} = (n, (-1)^{v-1/2}\lambda)_{p}\,.\]
When $v$ is even, we have already shown that $n = k - \lambda$ is the square of an integer, so 
all local invariants must vanish. 
\end{proof}

\begin{remark}\normalfont
By the Hasse local-global principle, the conditions $(n,(-1)^{(v-1)/2}\lambda)_p=1$ for all odd primes $p$, together with $n=k-\lambda>0$ (non-triviality of the design), imply that $z^2=nx^2+(-1)^{(v-1)/2}\lambda y^2,$
has a non-trivial \textbf{rational} solution. Multiplying by a common denominator of $x, y,$ and $z$, we find a non-trivial integral solution to the Diophantine equation
\[z^2=nx^2+(-1)^{(v-1)/2}\lambda y^2.\]
This is how the Bruck-Ryser-Chowla theorem is typically presented in the design theory literature. This formulation 
suggests to the reader the possibility of constructing explicit solutions to the given equation. These provide no insight 
into the existence of the design, and are typically much harder to compute than the Hilbert symbols themselves. 
\end{remark}

\subsection{Decomposition of symmetric designs} 

We consider the following question: when can the incidence matrix of a symmetric design be written as the sum of two disjoint $\{0,1\}$ matrices, each of which is the incidence matrix of a symmetric design? The obvious necessary condition is that designs with suitable parameters should exist individually. In this section, we develop a further necessary condition in terms of invariants of quadratic forms. 

\begin{proposition} 
Suppose that $M$ is the incidence matrix of a symmetric $(v,k,\lambda)$ design, 
and that $M= M_{1} + M_{2}$ where $M_{i}$ is the incidence matrix of a $(v, k_{i}, \lambda_{i})$ design. 

Then $k= k_{1} + k_{2}$ and $\lambda = \lambda_{1} + \lambda_{2} + \alpha$ where $\alpha = \frac{2k_{1}k_{2}}{v-1}$ is an integer. Furthermore, $M_{1}M_{2}^{\top} + M_{2}M_{1}^{\top} = \alpha (J-I)$. 
\end{proposition} 

\begin{proof} 
A standard counting argument establishes that $\lambda_{i} = \frac{k_{i}(k_{i}-1)}{v-1}$. Then 
\[\lambda = \frac{(k_{1} + k_{2})(k_{1}+k_{2})-1}{v-1} = \lambda_{1} + \lambda_{2} +  \frac{2k_{1}k_{2}}{v-1}\,.\] 
This establishes the claim about $\alpha$.

For the second claim, compute $(M_{1} + M_{2})(M_{1} + M_{2})^{\top} = M_{1}M_{1}^{\top} + M_{2}M_{2}^{\top} + M_{1}M_{2}^{\top} + M_{2}M_{1}^{\top}$ and use the formula for the Gram matrix of the incidence matrix of a symmetric design.
\end{proof} 

We will apply the theory of quadratic forms in essentially the same way as in the Bruck-Ryser-Chowla theorem. 
A quick computation shows that taking $X = M_{1}M_{2}^{\top}$ and computing necessary conditions for the existence of $X^{\top}X$ yields nothing in addition to the existence conditions for $M_{1}$ and $M_{2}$. So we are led to the following matrix. 

\begin{proposition} 
Let $Q = M_{1}M_{2}^{\top} + I$. Then $QQ^{\top} = \sigma I + \tau J$ where $\sigma = \left( k_{1} - \lambda_{1}\right)\left(k_{2} - \lambda_{2}\right) -\alpha + 1$ and $\tau = v \lambda_{1}\lambda_{2} +\lambda_{2} (k_{1} - \lambda_{1}) + \lambda_{1}(k_{2}-\lambda_{2})  + \alpha$. 
\end{proposition} 

\begin{proof} 
The result follows from computing $QQ^{\top}$ directly, observing that $M_{i}$ has constant row sum and so commutes with $J$, and substituting $\alpha (J-I)$ for $M_{1}M_{2}^{\top} + M_{2}M_{1}^{\top}$. 
\end{proof}

Note that the matrix $Q$ is normal (commutes with its transpose), either from a general result due to Ryser (Theorem 8.2.1 of \cite{Ryser}), or by computing $Q^{\top}Q$ directly. 

\begin{theorem}\label{BRCMosaic}
Suppose that $M = M_{1} + M_{2}$ is a decomposition of symmetric designs.
If $v$ is even then 
\[ \left( k_{1} - \lambda_{1}\right)\left(k_{2} - \lambda_{2}\right) -\frac{2k_{1}k_{2}}{v-1} + 1\] 
is the square of an integer. If $v$ is odd, then 
\[ (\sigma, \sigma)^{\binom{v-1}{2}}_{p} (\sigma, v)_{p} = (\sigma, (-1)^{v-1/2}v)_{p} = 1\] 
for all odd primes $p$. 
\end{theorem} 

\begin{proof} 
Since the matrix $M_{1}M_{2}^{\top} + I$ has constant row sum, it has the all-ones vector as an eigenvector, 
with eigenvalue $k_{1}k_{2} + 1$. Hence $MM^{\top}$ has constant row sum equal to $(k_{1}k_{2}+1)^{2}$. 
Comparing to the expression for $MM^{\top} = \sigma I + \tau J$, we find that 
\[ \sigma + v\tau = (k_{1}k_{2} + 1)^{2} \,.\] 
The determinant of $\sigma I_{v} + \tau J_{v}$ is $(\sigma + v\tau) \sigma^{v-1}$. Thus if $v$ is even, one requires 
that $\sigma$ is a square.

Since $\sigma + v\tau$ is square, the application of Proposition \ref{prop:combmat} with $a = \sigma + \tau$ and $b = \tau$ results in the expression
\[ (\sigma, \sigma)^{\binom{v-1}{2}}_{p} (\sigma, v)_{p} \,.\] 
As in the proof of the Bruck-Ryser-Chowla theorem, this may be further simplified by considering $v\mod 4$ and
using properties of the Hilbert symbol to achieve the claimed result.
\end{proof} 

A computer search for feasible parameters (i.e., satisfying integrality and Bruck-Ryser-Chowla conditions for each
 component design) shows very few feasible parameter sets. 
\begin{corollary} 
If $v$ is even, there is no decomposition of symmetric designs on less than 10,000 points. 
\end{corollary} 

\begin{proof} 
There is only one parameter set on less than $10,000$ points for which a decomposition could be possible: the trivial conditions for a $(2380, 976, 400)$-design to decompose into a $(2380, 183, 14)$-design and a $(2380, 793, 264)$-design are that $976 - 400 = 24^{2}$ and $183 - 14 = 13^{2}$ and $793 - 264 = 23^{2}$. But Theorem \ref{BRCMosaic} requires that $13^2\cdot 23^2 - 121 = 2^{6} \cdot 3^{2}\cdot 5\cdot 31$ be a perfect square, which it is not. 
\end{proof} 

In contrast, the conditions at odd orders are rather weaker. We observe that the incidence matrix of a $(91,81,72)$-design (the complementary design of a projective plane of order $9$) cannot be written as the sum of designs with parameters 
$(91, 36, 14)$-design and a $(91, 45, 22)$-design. The relevant parameters for the computation are
\[ k_{1} = 36, \,\, \lambda_{1} = 14, \,\,\, k_{2} = 45, \,\, \lambda_{2} = 22, \,\,\, \alpha = 36, \,\,\, \sigma = 471 \] 
The local invariants are $(471, 471)_{p}(471, 91)_{p}$ for all primes $p$. The prime $3$ divides $471$, so the invariant at $p = 3$ simplifies to $(3,3)_{p} (1, 3)_{p} = -1$. So Theorem \ref{BRCMosaic} shows that this decomposition does not exist. 

These methods \textbf{cannot} rule out the existence of a $(31, 25, 20)$-design (the complement of a projective plane of order 5) which decomposes into a $(31, 15, 7)$-design and a $(31, 10, 3)$-design. This is the smallest open case for a decomposition. Finally, we observe that solutions to this problem do exist, the sum of a skew-Hadamard design with parameters $(4t-1, 2t-1, t-1)$ with a trivial $(4t-1, 1,0)$-design gives a $(4t-1, 2t, t)$-design, so the concept is not vacuous \cite{mypaper-nesting}. This topic will be considered more fully in forthcoming work of the authors.

\subsection{Bose-Connor theorem} 

The Bose-Connor theorem gives non-existence conditions for \textit{group-divisible designs}. 

\begin{definition} 
Let $V$ be a set of size $mn$, divided into $m$ \textit{groups} of size $n$. 
Let $B$ be a set of blocks, each of size $k$. Then $(V, B)$ is a group-divisible 
design with parameters $(mn, n, k, \lambda_{1}, \lambda_{2})$ if any pair of points 
from the same group occurs in $\lambda_{1}$ blocks and any pair of points from distinct 
blocks occurs together in $\lambda_{2}$ blocks. 
\end{definition} 

Standard counting arguments show that each point appears in $r = \frac{(n-1)\lambda_{1} + n(m-1)\lambda_{2}}{k-1}$ blocks (and integrality of this quantity is a necessary condition for the existence of a group-divisible design). A group-divisible design is \textit{symmetric} if the incidence matrix is square, in which case $r = k$ as in the usual theory of symmetric designs. 

Recall that the Kronecker product of matrices $A$ and $B$ is given 
(as a block-matrix) by $[A \otimes B]_{ij} = a_{ij}B$. In particular, $I_{m} \otimes J_{n}$ is an $mn \times mn$ matrix with $n\times n$ blocks of ones on the diagonal and zeros elsewhere. It follows from the definition that the incidence matrix of a symmetric group-divisible design is $G = (r-\lambda_{1}-\lambda_{2})I + (\lambda_{1} - \lambda_{2}) I_{m} \otimes J_{n} + \lambda_{2}J_{mn}$. Non-existence conditions can be derived from the theory of quadratic forms by finding conditions under which $G$ is not a Gram matrix. We refer the interested reader to the original paper for a proof. 

\begin{theorem}[Bose-Connor, \cite{BoseConnor}] 
Suppose that $D$ is a symmetric group divisible design with parameters $(mn, n, k, \lambda)$, 
and denote $Q = k - \lambda_{1}$ and $P = k^{2} - v\lambda_{2}$. Then the following conditions hold: 
\begin{enumerate} 
\item $(n-1)\lambda_{1} + n(m-1)\lambda_{2} = k(k-1)$.
\item $P> 0$ and $Q > 0$. 
\item $P^{m-1}Q^{m(n-1)}$ is a perfect square. 
\item If $m$ is even then $P$ is a perfect square. If $m \equiv 2 \mod 4$ and $Q$ is even then $(Q, -1)_{p} = 1$ for all odd primes $p$. 
\item If $m$ is odd and $n$ is even then $Q$ is a perfect square. Furthermore, $( (-1)^{\binom{m}{2}}n\lambda_{2}, P)_{p} = 1$ for all odd primes $p$. 
\item If $m$ and $n$ are odd then $( (-1)^{\binom{m}{2}}n\lambda_{2}, P)_{p} = ( (-1)^{\binom{n}{2}}n, Q)_{p}$ for all odd primes $p$. 
\end{enumerate}
\end{theorem} 

\subsection{Maximal determinant matrices}

The \textit{Hadamard maximal determinant problem} asks for the maximal determinant of a $\{\pm 1\}$ matrix in dimension $n$. If $4$ divides $n$, then the optimal solution is a Hadamard matrix, provided that one exists. (This is the original motivation for the Hadamard conjecture.) No obstructions to the existence of Hadamard matrices arise from the theory of quadratic forms. The following result is compiled from the work of a number of mathematicians, principally Ehlich, Wojtas and Cohn. While the first two parts can be obtained without using quadratic forms, that theory allows for a uniform proof technique (essentially by computing the obvious invariants). 

\begin{theorem}[See \cite{MaxDet}]
Let $D_{n}$ be the absolute value of the maximal determinant of an $n \times n$ matrix with entries in $\{\pm 1\}$.
\begin{enumerate} 
\item If $n \equiv 1 \mod 4$ then $D_{n} \leq \sqrt{2n-1}(n-1)^{\frac{n-1}{2}}$. If the bound is met with equality, then $(n-1)I_{n}+ J_{n}$ is a Gram matrix and $2n-1$ is a perfect square. 
\item If $n \equiv 2 \mod 4$ then $D_{n}\leq (2n-2)(n-2)^{\frac{n-2}{2}}$. If the bound is met with equality then $I_{2} \otimes \left( (n-2)I_{n/2} + 2J_{n/2} \right)$ is a Gram matrix and $2n-2$ is the sum of two squares. 
\item If $n \equiv 3 \mod 4$ and $n \geq 63$ then $D_{n} \leq \frac{2 \cdot 11^{3}}{7^{\frac{7}{2}}} n(n-1)^{3}(n-3)^{\frac{n-7}{2}}$. If the bound is met with equality then $I_{7} \otimes \left( (n-3) I_{n/7} + 4J_{n/7} \right) - J_{n}$ is a Gram matrix. Letting $n = 7m$, this implies that $4m-3$ is a perfect square and that $(11m-3, -(7m-3))_{p} = 1$ for all odd primes $p$.
\end{enumerate} 
\end{theorem} 

The first two cases follow from requiring the discriminant of the Gram matrix to be a perfect square, in which case all local invariants become trivial as in the proof of the even case of Theorem \ref{BRC}. In the third case, the theory is rather more complicated and the optimal Gram matrix is not known to be realised by a $\{\pm 1\}$-matrix for any integer $n$. The conditions imposed by local invariants are non-trivial in this case however, and the smallest order which is not ruled out is $n = 511$, \cite{Tamura}. 

\section{Acknowledgements} 

P\'OC acknowledges support from the European Consortium of Innovative Universities; the Conference Participation Scheme of the Faculty of Humanities and Social Sciences at Dublin City University and a Learning Enhancement Project of the Technical University of the Shannon: Midlands Midwest. P\'OC thanks Drs Patrick Browne and Ronan Egan for delivering workshops at which the material of this paper was presented and the exposition refined. Financial support for the workshops from the School of Mathematical Sciences and Fiontar agus Scoil na Gaeilge, both at Dublin City University, is gratefully acknowledged. 

We thank the anonymous referee for a careful and thorough proof-reading and the Editor-in-Chief for his handling of the article. 

\bibliographystyle{abbrv}
\flushleft{
\bibliography{Biblio2020}
}

\end{document}